\documentclass[12pt]{article}
\usepackage{amsthm,amsmath,amssymb,mathrsfs,amsbsy}
\usepackage[utf8]{inputenc}
\usepackage[english]{babel}
\usepackage[colorlinks,linkcolor = red]{hyperref}
\usepackage[
backend=biber,
style=alphabetic,
sorting=ynt
]{biblatex}
 
\nocite{*}

\hypersetup{ linktoc=page }
\newtheorem{theorem}{Theorem}[subsection]
\newtheorem*{theoremA}{Theorem A}
\newtheorem*{theoremB}{Theorem B}
\newtheorem*{theoremC}{Theorem C}
\newtheorem*{theoremD}{Theorem D}
\newtheorem{definition}[theorem]{Definition}
\newtheorem{claim}[theorem]{Claim}
\newtheorem{lemma}[theorem]{Lemma}
\newtheorem{example}[theorem]{Example}
\newtheorem{corollary}[theorem]{Corollary}
\newtheorem{notion}[theorem]{Notion}
\newtheorem{convention}[theorem]{Convention}
\DeclareMathOperator{\spec}{Spec}
\DeclareMathOperator{\Sym}{Sym}

\catcode`\@=11
\newdimen\cdsep
\def\cdstrut{\vrule height .6\cdsep width 0pt depth .4\cdsep}
\def\@cdstrut{{\advance\cdsep by 2em\cdstrut}}

\def\arrow#1#2{
	\ifx d#1
	\llap{$\scriptstyle#2$}\left\downarrow\cdstrut\right.\@cdstrut\fi
	\ifx u#1
	\llap{$\scriptstyle#2$}\left\uparrow\cdstrut\right.\@cdstrut\fi
	\ifx r#1
	\mathop{\hbox to \cdsep{\rightarrowfill}}\limits^{#2}\fi
	\ifx l#1
	\mathop{\hbox to \cdsep{\leftarrowfill}}\limits^{#2}\fi
}
\catcode`\@=12

\title{Singularity Properties of Graph Varieties}
\author{Guy Kapon}
\newcommand{\PX}[2]{\widetilde{X}(#1,#2)}

\begin{document}

\maketitle

\begin{abstract}

For a graph $G=(V,E)$, and a symplectic vector space $(W, \left<\cdot,\cdot\right>)$, we define a variety $X(G,W)$ consisting of all functions $w:V\to W$ satisfying $\left<w(u), w(v)\right> = 0$ for any edge $\{u,v\}$ in $G$. We study the singularities of this varieties.




\end{abstract}

\tableofcontents

\section{Introduction}
Let $G = (V,E)$ be a graph, and $W$ a symplectic space. The graph variety of $G,W$, denoted $X(G,W)$, is the set of all assignments $f:V\rightarrow W$ s.t. for each edge $\{v,u\} \in E$ we have $ \left < f(v),f(u) \right > = 0$.

In \cite{AA1} these graph varieties were related to $$Def_{H,n}:=\{\left( g_{1},h_{1},...,g_{n},h_{n} \mid g_{i},h_{i}\in H, [g_{1},h_{1}]\cdot .... \cdot [g_{n},h_{n}]=1 \right) \}$$ where $H$ is an algebraic group over a field of characteristic zero, and were used to prove that for $n$ big enough, it has rational singularities. This in turn was used in \cite{AA2} to give bounds on the growth of irreducible representations of algebraic groups (see Theorem \ref{App1}).

There are natural questions about this graph varieties. For starters, we could ask what are their dimension and are they irreducible?

The graph variety is always singular at zero, but what if we consider the projective version of it, where vectors are replaced by lines? We can relate the singularity of the graph variety with the properties of the projective version, so a natural question is when is the projective graph variety is smooth.

Finally, it's known that for a fixed graph $G$, and for any vector space $W$ with high enough dimension, $X(G,W)$ has rational singularities. It's interesting to see if we can improve those bounds. 


We solve the first two questions when the vector space is large enough, and completely for trees, and improve the results for the third question.

\subsection{Main Results}
For a graph $G$ and a vector space $W$ we denote our graph variety $X(G,W)$. We first prove that it is usually irreducible.

A graph is called $\textsl{d-degenerate}$ if it can be built by consecutively adding vertices and connecting them to at most $d$ other already existing vertices (for example a tree is $1$-degenerate), for a more precise definition see section \ref{subsec:2.4}.

Our first result is concerned with the dimension and irreducibility of $X(G,W)$. 
\begin{theoremA}
Let $G=(V,E)$ be a $d$-degenerate graph with maximal degree $D$, and $(W,\left\langle \cdot , \cdot \right\rangle )$ a vector space with a symmetric or anti-symmetric non-degenerate form $\left < \cdot , \cdot \right >$. If $\dim(W)\geq D+d$, then $X(G,W)$ has dimension $\dim(W)\left|V \right| -\left|E \right| $ and is irreducible.
\end{theoremA}

Note that, in particular, if we think of $X(G,W)$ as the fiber of zero in the map $\phi:W^{V} \rightarrow \mathbb{F}^{E}$ defined by $\phi(f)(\{v,u\}) = \left < f(v), f(u) \right> $, we get that under the theorem conditions that this map is flat at $\phi^{-1}(0)$.

Since $X(G,W)$ is never smooth (for example the point where all vectors assigned are zero is always singular) we then look at the projective version of $X(G,W)$, which we denote by $\PX{G}{W}$. Its definition is the same as $X(G,W)$ only we assign lines in $W$ instead of vectors. It can also be defined as the blow-up of $X(G,W)$ at the sub-variety where some vector is zero.

\begin{theoremB}
	If $\dim(W)\geq D+d$ then $\PX{G}{W}$ is smooth if and only if $G$ is a forest.
\end{theoremB}    

Finally we use these to prove:
\begin{theoremC}
	If $G$ is a forest and $\dim(W)\geq D+1$ then $X(G,W)$ has rational singularities.
\end{theoremC}

    Also on a different front we prove a combinatorial result, that together with \cite[Corollary 2.4.9]{AA1} proves:
\begin{theoremD}
	If $\dim(W)\geq 2\cdot p_{4}(D)$ where $p_{4}(D)=\frac{D^{2}(D+1)^{2}}{2}-1$, then $X(G,W)$ has rational singularities.  
\end{theoremD}

\subsection{Applications}
Using the improved bounds of \hyperref[thmC]{Theorem C} for rational singularities for $X(G,W)$, and \cite[Sections 2.3 and 2.4]{AA1}:
\begin{theorem}
\label{thmR}
	For a simple Lie algebra $\mathfrak{g}$ over an algebraically closed field of characteristic zero, let:
	$$\mathfrak{B}(\mathfrak{g})=\begin{cases} 10 & \mathfrak{g}=\mathfrak{sl}_{d} \text{ or } \mathfrak{g}=\mathfrak{so}_{d} \\ 20 & \mathfrak{g}=\mathfrak{sp}_{d} \\ 3\dim(\mathfrak{g})+1 & \mathfrak{g} \text{ is exceptional} \end{cases}$$
	Now suppose that you have a semi-simple algebraic group $H$ over a field of characteristic zero, and $n\geq \mathfrak{B}(\mathfrak{g})/2+1$ for any simple factor of $Lie(H) \otimes \bar{k}$ then $Def_{H,n}$ has rational singularities. 
\end{theorem} 

Before it was known for: 	$$\mathfrak{B}(\mathfrak{g})=\begin{cases} 22 & \mathfrak{g}=\mathfrak{sl}_{d} \text{ or } \mathfrak{g}=\mathfrak{so}_{d} \\ 40 & \mathfrak{g}=\mathfrak{sp}_{d} \\ 3\dim(\mathfrak{g})+1 & \mathfrak{g} \text{ is exceptional} \end{cases}$$

Further more \ref{thmR} strengthens the result from \cite[Theorem B]{AA2}:
\begin{theorem}
\label{App1}
	If $G$ is an algebraic group scheme whose generic fiber $G_{\mathbb{Q}}$ is simple, connected, simply connected, and of $\mathbb{Q}$-rank at least two, and $C\geq \mathfrak{B}(\mathfrak{g})$ for any simple factor of $Lie(H) \otimes \bar{k}$ then:
	$$ \left|  \begin{Bmatrix} \text{Irreducible represntaions of } H(\mathbb{Z}) \\ \text{of dimension at most }n \end{Bmatrix}  \right| = O(n^{C}) $$
\end{theorem}
\subsection{Ideas of The Proofs}
Since our variety has a geometric interpretation, it has nice projections to varieties of the same kind (by discarding one of the vectors). For some nice part of our variety, the fibers are irreducible and of constant dimension. These projections maps allow us to calculate the dimension and irreducibility of this part using elementary algebraic geometry tools. Then we bound the remaining part of the variety and when it's small enough we extend our result to the whole variety.

Then we derive our equations, and get a combinatorial description for the singular points of the projective version. Using this, we understand when is it smooth.

Finally, when the graph is a tree we get an explicit resolution of singularities. Using a calculation of the canonical bundle and the Kodaira vanishing theorem we show that $X(G,W)$ has rational singularities.

For the combinatorial part, the degeneration method used showed you can split the graph into certain sub graphs, and it's enough to show that they have rational singularities. However, you have to split the vector space in the process, so the more you split the graph, the bigger the vector space needs to be. It is shown directly that for an edge the singularities are rational, so by splitting the graph to disjoint edges we can get a bound for when the singularities are rational. The proof uses repeated splittings of the graph, but it turns out that doing this in one big splitting with lots of colors is more effective.
\subsection{Structure of the paper}
In section \ref{sec:2} we give a number of preliminaries, definitions and notions we will use.

In \ref{subsec:2.1} we define the graph variety and the expected dimension.

In \ref{subsec:2.2} we recall some definitions and theorems regarding the canonical bundle.

In \ref{subsec:2.3} we recall the definition of the higher push-forward map and the definition of rational singularities.

In \ref{subsec:2.4} we give some notions we will use for graphs.

In \ref{sec:3} we study the graph variety. 

In \ref{subsec:3.1} we prove \hyperref[thmA]{Theorem A}, by first showing it for a certain part called the $>$-regular part, where $>$ is an order on the vertices, and then extend to the whole variety.

In \ref{subsec:3.2} we understand the smooth points of the variety to some extent, and prove \hyperref[thmB]{Theorem B}.

In \ref{subsec:3.3} we calculate the canonical bundle for the projective variety when the graph is a tree.

In \ref{sec:4} we prove the results about the rational singularities for trees.
In \ref{subsec:4.1} we prove \hyperref[thmC]{Theorem C}.

In \ref{subsec:4.2} we prove the combinatorial result i.e. \hyperref[thmD]{Theorem D}.

\subsection{Acknowledgments}

I would like to thank my advisor \textbf{Avraham Aizenbud}, for suggesting this problem, helping me along the way and teaching me mathematics. 

I would also like to thank \textbf{Itay Glazer} for providing helpful comments on this paper.

I would like to thank my friends who studied with me and motivated me, among which are \textbf{Daniel Kaner},  \textbf{Dor Mezer}, \textbf{Guy Shtotland}, \textbf{Lev Radzivilovsky} and \textbf{Yaron Brodsky}.

In addition I would like to thank \textbf{Nathan Keller} and \textbf{Uzi Vishne}, for supporting me in the early days of my studies.

I also deeply thank \textbf{Yaron Brodsky} for helping proof-read this paper, both grammatically and mathematically.

Finally I would like to thank \textbf{Lev Radzivilovsky} and \textbf{Shachar Carmeli} for teaching me a huge amount of mathematics. 
\section{Preliminaries}
\label{sec:2}
\subsection{The Graph Variety}
\label{subsec:2.1}
In most cases, the graph variety arises naturally with a symplectic space. However, almost all our proofs will work for symmetric non-degenerate bi-linear forms as well, and we will only mention if the form is symmetric or anti-symmetric if there is a difference.
\begin{convention}
	A vector space $W$ will always have a non-degenerate bi-linear form on it, either symmetric or anti-symmetric.
	\end{convention}  
\begin{definition}
	For a graph $G=(V,E)$ and vector a space $(W,\left<\cdot,\cdot\right>)$, $X(G,W)$ is the sub-variety of $W^{V}$ defined by the equations $\left< w(v),w(u)\right>=0$ for $\{v,u\}\in E$.
\end{definition}
We will often denote a point of $X(G,W)$ by $w$ and think of it as a function from $V$ to $W$, i.e., for a vertex $v$, $w(v)$ is the vector assigned to $v$, and $w(V)$ is the set of all vectors assigned to any vertex.

Since $X(G,W)$ is cut from $W^{V}$ by $\left | E \right |$ equation, we except the dimension to be $\left| V\right| \cdot \dim(W)-\left| E\right|$, and we define:
\begin{definition}
\label{def1}
	The expected dimension of $X(G,W)$ is $d(G,W)=\left| V\right| \cdot \dim(W)-\left| E\right|$. 
\end{definition}
We also consider the projective version of this scheme.
\begin{definition}
	For a graph $G=(V,E)$ and vector a space $(W,\left<\cdot,\cdot\right>)$, $\PX{G}{W}$ is the sub-scheme of $\mathbb{P}(W)^{V}$ defined by the equations $\left< w(v),w(u)\right>=0$ for $\{v,u\}\in E$.
\end{definition}
\subsection{Canonical Bundle}
\label{subsec:2.2}
We recall some definition and theorems about the canonical bundle.
\begin{definition}
	For a $k$-algebra $A$ we define $\Omega (A)$ to the $A$ module generated by $df$ for $f\in A$, satisfying the relations: $$d(f+g)=df+dg$$ $$dfg=fdg+gdf$$ $$d\alpha f=\alpha df$$ For $f,g \in A$, $\alpha \in k$.  
	The co-tangent bundle of a smooth variety $X$ is the sheaf satisfying $\Omega_{X}(O)=\Omega (\mathcal{O}_{X}(O))$ for any affine open subset $O$.
	To construct this explicitly we look at the diagonal embedding $\Delta:X\rightarrow X\times _{k} X$ and if $\mathcal{I}$ is the ideal sheaf of the diagonal, then: $$ \Omega_{X}=\Delta^{*}(\mathcal{I}/\mathcal{I}^{2})$$
\end{definition}
\begin{definition}
	The canonical bundle of a smooth variety $X$, denoted $\omega_{X}$, is the top exterior product of the co-tangent bundle.
\end{definition}
We recall the canonical bundle of some basic varieties.
\begin{example}
	The canonical bundle of $X=\mathbb{P}^{n}$ is $\mathcal{O}_{X}(-n-1)$.
\end{example}
\begin{example}
	If $X,Y$ are smooth varieties and $p_{1},p_{2}$ are the projections from $X\times Y$ to $X$ and $Y$ respectfully then the canonical bundle of $X\times Y$ is given by: $$\omega_{X\times Y} = p_{1}^{*} \omega_{X} \otimes p_{2}^{*} \omega_{Y}$$ 
\end{example}
Finally we write the adjunction formula.
\begin{theorem}
	If $X$ is a smooth variety, $D$ a smooth divisor, and $i:D\rightarrow X$ the inclusion, then the canonical bundle of $D$ is given by: $$\omega_{D}=i^{*}(\omega_{X}\otimes \mathcal{O}(D))$$  
\end{theorem}
\subsection{Rational Singularities and $R^{q}f_{*}$}
\label{subsec:2.3}
We recall some definitions and properties. During this subsection $f:X\rightarrow Y$ is a morphism of schemes.
\begin{definition}
	The push-forward morphism is a functor between the category of $\mathcal{O}_{X}$-sheaves on $X$ to $\mathcal{O}_{Y}$-sheaves on $Y$, and furthermore it is left exact. 
	We define $R^{q}f_{*}$ as its right derived functor. It follows that $R^{0}f_{*}=f_{*}$.
\end{definition} 
\begin{example}
	If $f:X\rightarrow \spec{k}$ is a $k$-scheme, then $f_{*}$ is just the global sections functor, and therefore $R^{q}f_{*}=H^{q}$ is the usual cohomology.
\end{example}
\begin{example}
	If $X=\spec(A)$ is affine, where $A$ is a $k$-algebra, and $f:X\rightarrow \spec{k}$ is the corresponding map, then on quasi-coherent sheaves $f_{*}$ is an equivalence between $\mathcal{O}_{X}$-sheaves and $A$ modules, and therefore exact, so $R^{q}f_{*}=0$ for $q>0$ and quasi-coherent sheaves. 
\end{example}
\begin{example}
    More generally, if $f:X \to Y$ is a map between affine schemes, then $f_{*}$ is exact on quasi-coherent sheaves.
    
\end{example}
\begin{example}
	If $f:X\rightarrow Y$ is an affine map then, since the inverse image of an affine subset is affine, $f_{*}$ is exact on an affine cover of $Y$, and therefore $R^{q}f_{*}=0$ for $q>0$ on quasi-coherent sheaves.
\end{example}
We also note an obvious claim:
\begin{claim}
	If $f:X\rightarrow Y,g:Y\rightarrow Z$ are morphisms and $f_{*}$ is exact, then $R^{q}(gf)_{*}=(R^{q}g_{*})f_{*}$
\end{claim}
Finally we give a definition of rational singularities and state the Kodaira vanishing theorem:
\begin{definition}
	A scheme $X$ of finite type over a field of characteristic $0$ has rational singularities if it is normal and there exists a smooth scheme $Y$ and a proper bi-rational map $f:Y\rightarrow X$ s.t. $R^{q}f_{*}(\mathcal{O}_{Y})=0$ for $q>0$. 
\end{definition}  

\begin{theorem}
	(Kodaira Vanishing Theorem): Let $X$ be a smooth projective scheme over a field of characteristic zero, $\omega_{X}$ the canonical bundle, and $\mathcal{L}$ an ample line bundle, then: $$H^{q}(X,\omega_{X}\otimes \mathcal{L})=0$$
	For $q>0$.
\end{theorem}
\subsection{Some Graph Defintions}
\label{subsec:2.4}
While for the definition of graph variety we only need a graph and a symplectic space, it will be convenient for the proof to order the vertices of the graph. In practice any order can be given, but different orders may yield better results.
\begin{definition}
	An ordered graph $G=(V,E,>)$ is a graph with a linear ordering $>$ on its vertices.
\end{definition}
We will usually assume $V=\{v_{1}<v_{2}<...<v_{r}\}$. For a graph and a vertex $v$, $N(v)$ is the set of neighbors of $v$, and similarly for ordered graphs we have an analogue:
\begin{definition}
	Let $G=(V,E,>)$ be an ordered graph, $v \in V$ a vertex. The set of the older neighbors of $v$ is $N_{>}(v):=\{u\in V \mid u>v \ and \ \{v,u\}\in E \}$.
\end{definition}
\begin{definition}
	An ordered graph is said to be $d$-degenerate if $\left| N_{>}(v)\right| \leq d$ for all vertices $v$. A graph is said to be $d$-degenerate if its vertices can be ordered in a way that makes it a $d$-degenerate ordered graph.
\end{definition}
In other words, ordered graphs are built by adding one vertex at a time, and are $d$-degenerate if every time we add a vertex we connect him to at most $d$ already existing vertices.
\begin{example}
	Since any tree has at least one leaf, it's easy to see that a tree can be ordered as a $1$-degenerate graph. On the other hand a connected $1$-degenerate is obviously a tree.
\end{example}
\begin{definition}
	If $G=(V,E)$ is a graph, and $U \subseteq V$ is a subset, then $G\mid_{U}$ is the induced graph on $U$. 
\end{definition}

\section{Basic Properties of $X(G,W)$ And $\PX{G}{W}$}
\label{sec:3}
For the rest of the section let $G = (V,E)$ be a graph, $(W,\left < , \right >)$ be a vector space, and let $X(G,W)$ ($\PX{G}{W}$) be the corresponding (projective) graph variety. 
\subsection{Dimension And Irreducibility}
\label{subsec:3.1}
Our first goal is to give a sufficient condition on $(G,W)$ s.t. $X(G,W)$ will be irreducible and will have the expected dimension (see definition \hyperref[def1]{\ref*{def1}}).

It is possible to show that you can degenerate one graph variety to another one with the same graph but with a smaller vector space. This implies that as the dimension of $W$ grows, the singularities of $X(G,W)$ become nicer. With this in mind we give the following theorem we shall prove in this section.
\begin{theoremA}
\label{thmA}
	Let $G=(V,E)$ be a a $d$-degenerate graph with maximal degree $D$, and $(W,\left<\cdot,\cdot\right>)$ a vector space. If $\dim(W)\geq d+D-1$. Then $X(G,W)$ has dimension $d(G,W)$. If $\dim(W) > d + D - 1$ then $X(G,W)$ is irreducible. 
\end{theoremA}
The proof will be split into two parts, starting with showing the result for a somewhat regular part of the variety, and then showing that for $\dim(W)$ large enough this part is indeed most of the variety.
\subsubsection{The $>$-Regular Part}
Intuitively the variety structure is quite simple - each vector is taken from some subspace dependent on its neighbors. The problem arises when those vectors are linearly dependent. We solve this by forcibly taking those points out.

\begin{definition}
	Let $G=(V,E,>)$ be an ordered graph with $V=\{v_{1}<v_{2}<...<v_{r}\}$, and $(W,\left<,\right>)$ a vector space. For $i\in \{1,2,...,r\}$, denote $A_{i}=\{v_{i},v_{i+1},...v_{r}\}$. We denote by  $U_{i}(G,W,>)$ the open subset of $X(G\mid_{A_{i}},W)$ s.t. for each $v\in V$ the set of vectors $w(N_{>}(v) \cap A_{i})$ is linearly independent. We call $U_1(G, W, >)$ the $>$-regular part of $X(G, W)$.
\end{definition}
A simple but important observation is that if the graph is $d-degenerate$ then the open conditions only states that some subsets of up to $d$ vectors are linearly independent, since $\left | N_{>}(v) \cap A_{i} \right | \leq \left | N_{>}(v) \right |\leq d$. Therefore if $d-1$ vectors from that set are already chosen, then we only ask that the $d$-th vector is in the complement of some $d-1$ dimensional vector space. 
\begin{claim}
\label{clm1}
	 Let $G=(V,E,>)$ be an ordered $d$-degenerate graph, and suppose that $\dim(W)\geq 2d$. Then $U_{i}(G,W,>)$ is non-empty, with dimension $d(G\mid_{A_{i}},W)$, and irreducible, for all $i\in \{1,2,...,r\}$.
\end{claim}
The proof of the claim is by induction. The idea is that for $i\in \{1,2,...,r-1\}$  there is a natural morphism from $U_{i}(G,W,>)$ to $U_{i+1}(G,W,>)$ by forgetting $w(v_{i})$. The claim will follow by considering the fibers of these maps.   

\begin{lemma}
	Let $G = (V,E,>)$ and $W$ be as in claim 3.1.3. If $i \in \{1,2,...,r-1\}$, and $f_{i}:U_{i}(G,W,>)\rightarrow U_{i+1}(G,W,>)$ is the morphism defined by forgetting the vector $w(v_{i})$, then for every $w\in U_{i+1}(G,W,>)$ the fiber $f_{i}^{-1}(w)$ is irreducible and has dimension $\dim(W) -  \left| N_{>}(v_{i})\right|$
\end{lemma}
\begin{proof}
	By definition, at the point $w$ the collection $w(N_{>}(v_{i}))$ is linearly independent.
	
	 We know that $w(v_{i})$ must be orthogonal to all these vectors, so each one gives a linear condition on $w(v_{i})$, and by assumption these conditions are independent. We conclude that $w(v_{i})$ must lie in a subspace $W'\subseteq W$ of co-dimension $  \left| N_{>}(v_{i})\right|$.
	 
	  The only other condition on $w(v_{i})$ is that he will be linearly independent with certain sets of vectors of size at most $d-1$, so there are a finite number of subspace of dimension at most $d-1$ s.t. $w(v)$ is in their complement. Since $\dim(W')=\dim(W)- \left| N_{>}(v_{i})\right| \geq 2\cdot d - d>d-1$ we get that these subspaces cannot cover $W'$, and therefore the fiber $f^{-1}(y)$ is a nonempty open subset of $W'$, and so irreducible and of dimension $\dim(W) -  \left| N_{>}(v_{i})\right|$
\end{proof}
This gives an obvious corollary.  
\begin{corollary}
	The dimension of $U_{i}(G,W,>)$ is $d(G\mid_{A_{i}},W)$.
\end{corollary}
\begin{proof}
	If $i=r$, then obviously $U_{r}(G,W,>)=W\setminus\{0\}$, and so has dimension $\dim(W)$ which is the expected dimension.
	
	Continuing by induction, if $i\in \{1,2,...,r-1\}$, and the result is proved for $i+1$, then since all fibers of $f:U_{i}(G,W,>)\rightarrow U_{i+1}(G,W,>)$ have dimension $ \dim(W) -  \left| N_{>}(v_{i})\right|$ we have: $$\dim(U_{i}(G,W,>))=\dim(W) -  \left| N_{>}(v_{i})\right|+d(G\mid_{A_{i+1}},W)=d(G\mid_{A_{i}},W)$$
\end{proof}
Recall that saying that a variety is a complete intersection is saying that its co-dimension is the size of some set of equations defining it. Note that the corollary implies that $U_i(G, W, >)$ is a complete intersection.

When we intersect some algebraic set $X$ with the zero locus of a polynomial $V$, then for each irreducible component $Z$ of $X$ either $Z\subseteq V$, $Z\cap V = \varnothing$, or $Z \cap V$ is a union of a finite number of irreducible components with co-dimension one in $Z$.
If a variety $X$ is a complete intersection, and we add equations defining it one by one, then if at some point an irreducible component remains as a whole, then when we add some other equation it must vanish, since otherwise its dimension will be too large in the end. As a corollary, all irreducible components of $X$ will have the same dimension. 
Since in our case we showed that $U_{i}(G,W,>)$ is a complete intersection, this is true for it as well. 
Now claim \ref{clm1} will follow easily.
\begin{proof}
	We already saw that if $i=r$ then $U_{r}(G,W,>)=W\setminus \{0\}$ is irreducible.
	Now let $i\in \{1,2,...,r-1\}$, and assume by induction that $U_{i+1}(G,W,>)$ is irreducible. We consider the same morphism $f_{i}:U_{i}(G,W,>)\rightarrow U_{i+1}(G,W,>)$ as before, and prove that $U_{i}(G,W,>)$ is irreducible.
	
	Let $Z$ be an irreducible component of $U_{i}(G,W,>)$ with maximal dimension, and look at $f_{i}|_{Z}:Z\rightarrow U_{i+1}(G,W,>)$. Since $$\dim(Z)=\dim(W) -  \left| N_{>}(v_{i})\right|+\dim(U_{i+1}(G,W,>)),$$ every non-empty fiber of this map must have dimension at least $ \dim(W) -  \left| N_{>}(v_{i})\right|$, on the other hand since they are contained in the fibers of $f$, their dimension is at most $\dim(W) -  \left| N_{>}(v_{i})\right|$, and so it is always equal $\dim(W) -  \left| N_{>}(v_{i})\right|$. Together with the fact that the fibers are irreducible and $Z$ is closed, this shows that if a fiber of $f$ intersects $Z$ then he is contained in it.
	
	Now notice that $\overline{f(Z)}=U_{i+1}(G,W,>)$ since if it were smaller $Z$ would obviously have smaller dimension. By Chevalley's theorem this implies $f(Z)$ contains an open set $U$.
	
	Now let $Z'$ be another irreducible component of $U_{i}(G,W,>)$ (if exists). For any point $z\in Z'$, either $f(z)\in U^{c}$ or $f(z)\in U$. In the second case, the argument before shows that $z\in Z$ so that $z\in Z\cap Z'$.
	
	Putting this two together we get $Z'\subseteq f^{-1}(U^{c})\cup (Z\cap Z')$, this however obviously has dimension less then the dimension of $Z$, and therefore so does $Z'$.
	
	This however contradicts the equidimensionality of  $U_{i}(G,W,>)$.

	In conclusion there can only be one irreducible component of $U_{i}(G,W,>)$ so we are done. 
\end{proof} 
\subsubsection{Proof of Theorem A}
To deduce \hyperref[thmA]{Theorem A} from claim \ref{clm1} all we have to show is that the complement of $U_{1}(G,W,>)$ in $X(G,W)$ has dimension at most $d(G,W)$. This will imply that $\dim(X(G,W)) = d(G,W)$. Furthermore if we show that the dimension of the complement is less then $d(G,W)$ then $X(G,W)$ will be irreducible, since the same argument as before shows that $X(G, W)$ is equidimensional.
\begin{proof}[proof of theorem A] 
	By definition, for every $w \in U_{1}(G,W,>)^{c}$ there exists $v, v_1, ...,v_k \in V$, where $k \leq d - 1$, s.t. $w(v)$ is a linear combination of $w(v_1), \dots, w(v_k)$.
	
	 It's enough to prove the theorem for the locally closed set where $w(v)$ is a linear combination on $w(v_1), \dots, w(v_k)$. Denote this set by $Y$.  
	 
	 Consider the projection $\pi:Y\rightarrow X(G\setminus v,W)$ defined by forgetting $w(v)$. By induction $X(G \setminus v,W)$ has dimension $d(G\setminus v,W)$, and since $w(v)$ is linearly dependent on some $w_{i_{1}},...,w_{i_{k}}$, all fibers have dimension at most $k\leq d-1$.
	 
	 This gives $\dim(Y)\leq d-1 +d(G\setminus v,W)$. Now our assumption that $\dim(W)\geq d + D - 1$ implies that: $$d(G,W) = \dim(W)-\left| N(v)\right| + d(G\setminus v,W) \geq $$ $$\geq \dim(W)-D+d(G\setminus v,W)\geq $$ $$ \geq d-1+d(G\setminus v,W)\geq \dim(Y)$$
	In addition if $\dim(W) > d + D - 1$, then $\dim(Y)<d(G,W)$ and we are done.
\end{proof}
The following two examples show that the bounds on $\dim(W)$ in theorem A are sharp.
\begin{example}
	Let $G$ be a tree with maximal degree $D$, and let $W$ be a vector space of dimension $n$.
	
	 Let $Z$ be the subvariety of $X(G,W)$ defined by $w(v)=0$ for a vertex $v$ with maximal degree. Notice that this is isomorphic to $X(G \setminus v,W)$, and hence $\dim(Z)$ is at least $d(G\setminus v, W)$.
	 
	 Since $$d(G\setminus v,W)-d(G,W)=D-n$$ If $n \leq D$ the dimension of this part is at least $d(G,W)$. However, the dimension of the $>$-regular part is $d(G, W)$, and it can be easily deduced that $X(G, W)$ is reducible.
	 
	If $n<D$ we also get that $X(G,W)$ has dimension bigger then $d(G,W)$.
	\end{example}  
\begin{example}
	Let $G=K_{d,D}$ be the complete bipartite graph on $d,D$ vertices with $D\geq d$, its a $d$-degenerate graph with maximal degree $D$.
	
	Let $W$ be a symplectic space with dimension $n$.
	The expected dimension is $d(G,W)=nd+nD-dD$.
	
	 We look at the part of $X(G,W)$ where out of the $d$ vertices on the smaller side, $d-1$ of them have linearly independent vectors, and the last one is spanned by the others. It has dimension: $$n(d-1)+d-1+(n-(d-1))D = nd+nD-dD+d+D-1-n$$ so if $n\leq d+D-1$ that part has dimension at least $d(G,W)$, and so if $n\geq 2d$ using the $>$-regular part we get that $X(G,W)$ is reducible. 
	 
	If $n<d+D-1$ we get that the dimension of $X(G,W)$ is not $d(G,W)$.   
\end{example}
In fact, looking over the proof of dimension and irreducibility for the $>$-regular part, we see that the last example is reducible for $n\geq d$ (since the result of that part still holds), for smaller $n$ the $>$-regular part is obviously empty, and what we did gives no information on the variety. 
\subsection{Smoothness}
\label{subsec:3.2}
Now that we know the dimension of the variety, we can try and find a resolution of singularities (the variety itself is of course singular). Our best bet is the projective version of it:
\begin{definition}
		For a graph $G=(V,E)$ and a vector space  $(W,\left<\cdot,\cdot\right>)$, $\widetilde{X}(G,W)$ is the sub-scheme of $\mathbb{P}(W)^{V}$ defined by the equations $\left< w(v),w(u)\right>=0$ for $\{v,u\}\in E$.
\end{definition}
The main difference between our original graph variety and the projective version is points for which $w(v)=0$ for some $v$. If $O(G,W)$ is the open set of $X(G,W)$ where $w(v)\neq 0$ for all $v$, then there is a natural projection $f:O(G,W)\rightarrow \widetilde{X}(G,W)$. All of its fibers are smooth and irreducible with the same dimension, so those two schemes have similar properties. In particular one is smooth if and only if the other is.

To understand when $\PX{G}{W}$ is smooth, we will want to derive our equations. If $e:V\rightarrow W$ is a direction vector, and $w:V\rightarrow W$ a point of $O(G,W)$ then plugging $w+e$ in the equations we get $\left<w(v)+e(v),w(u)+e(u)\right>=\left<w(v),w(u)\right>+\left<e(v),w(u)\right>+\left<w(v),e(u)\right>+\left<e(v),e(u)\right>$.
Eliminating the second degree term we get that the derivative is $\left<e(v),w(u)\right>-\left<e(u),w(v)\right>$ if the form is symplectic, and $\left<e(v),w(u)\right>+\left<e(u),w(v)\right>$ if it is symmetric.

If $\dim(W)\geq d+D-1$ we know that $X(G,W)$ has the expected dimension, so checking if it is smooth at a point is the same as saying the linear conditions obtained by deriving the equations are linearly independent, so we arrive at the following claim:
\begin{claim}
	For a graph $G=(V,E)$ and a symplectic space $W$, if $\dim(W)$ is big enough, then a point $w:V\rightarrow W$ is singular if and only if there is some non-zero weighting of the edges $\lambda:E\rightarrow \mathbb{C}$ and a directing of the edges such that for each vertex $v$:$$\sum_{\{v,u\}\in E} \epsilon_{\{v,u\}}\cdot \lambda(\{v,u\}) \cdot w(u)=0$$ where $\epsilon_{u}$ is $1$ if the edge is directed toward $u$ and $-1$ otherwise. If the form is symmetric then it is the same but without the $\epsilon$.     
\end{claim}    
This observation shows that the singularities of $X(G,W)$ are in some sense monotone - more explicitly, let $H$ is a subgraph of $G$, and $w:V(H)\rightarrow W$ is a singular point of $X(H,W)$ with a corresponding function $\lambda:E(H)\rightarrow \mathbb{C}$. Consider some point of $X(G,W)$ such that on the vertices of $H$ it is the same as $w$. Weight the edges in $H$ using $\lambda$, and give all the other edges weight $0$. This can easily be seem to satisfy the conditions of the claim, and so this is a singular point. 
With this in mind, we give the next claim:
\begin{claim}
\label{clm2}
	If $G$ is a circle , $W$ a symplectic space with $\dim(W)\geq 4$, then $O(G,W)$ is singular. If the circle has even length, and the form is symmetric then $O(G,W)$ is singular as well. 
\end{claim}    
\begin{proof}
	For the symplectic case, let $w\in W$ be a non-zero vector in $W$, we give all the vertices of $G$ the vector $w$, this is a point of $O(G,W)$ because $W$ is symplectic. 
	We direct the edges in a circle, and weight them all the same, This is easily seen to work.
	
	In the other case, pick some vector with $\left\langle v,v\right\rangle =0$, assign to all vertices this vector and weight the edges $\pm 1$ alternatively.
\end{proof}
Using this claim and the discussion above we get (only for symplectic spaces):
\begin{theoremB}
\label{thmB}
	If $\PX{G}{W}$ is smooth for $\dim(W)\geq d+D-1$, then $G$ is a forest.
\end{theoremB}  
For trees however it is easily seen to be true, In fact we prove a stronger claim.
\begin{claim}
	Let $G=(V,E,>)$ be an ordered graph, then all the $U_{i}(G,W)$ are smooth.
\end{claim}

Since for the case of a tree $O(G,W)$ is an open subset of $U_{1}(G,W)$ this implies the result.

\begin{proof}
	We prove that there is no weighting of the graph satisfying the condition. 
	For the first vertex $v$, we get that the sum of its neighbors' vectors with some weights must be zero, however his neighbors are linearly independent so all of his weights must be zero. Thus $v$ does not effect the existence of a weighting, and so we can look at $G\setminus v$ instead. By induction we are done. 
\end{proof}

For the symmetric case a few more graphs have smooth projective varieties. In the same way as claim \ref{clm2} one can show that if $G$ contains two connected circles then $\PX{G}{W}$ is singular. The remaining graphs are trees with one of their vertices enlarged to an odd circle, and for them it's easy to show that $\PX{G}{W}$ is indeed smooth. We omit the proof as it is very similar to the symplectic case, and we will not use it.

\subsection{The Canonical Bundle}
\label{subsec:3.3}
Using the previous subsection, we calculate the canonical bundle of $X(G,W)$ for the case where $G$ is a forest.
First, a notion:
\begin{notion}
	We denote $\mathcal{O}_{X,v}(n)$ the bundle on $X\subseteq \mathbb{P}(W)^{V}$ of degree $n$ homogeneous polynomial in the coordinates of $w(v)$, i.e if $f:X\rightarrow \mathbb{P}(W)$ is the projection on $w(v)$ then this is $f^{*}\mathcal{O}_{\mathbb{P}(W)}(n)$.
\end{notion}
\begin{theorem}
	If $G$ is a forest, and $\dim(W)=n$, then the canonical bundle of $\PX{G}{W}$ is: $$\bigotimes_{v\in V} \mathcal{O}_{X(G,W),v}(-n+\left| N(v)\right|  )$$
\end{theorem}
This will follow easily from the adjunction formula and the fact the singularities of the variety are monotone - this means that if we add the equations of the edges one by one, we will get a smooth variety at every step, and so we can use the adjunction formula.
\begin{proof}
	We prove by induction on the number of edges, if the graph is empty, then $\PX{G}{W}=\mathbb{P}(W)^{V}$ and since the canonical of $P(W)$ is $\mathcal{O}_{\mathbb{P}(W)}(-n)$, we get the result for the empty graph.
	
	Let us pick any edge $e=\{u_{1},u_{2}\}$ of $G$, and look at $G\setminus e$, by induction the theorem is true for $G\setminus e$, notice that $X(G,W)$ is obtained from $X(G\setminus e,W)$ by adding the equation $\left<w(u_{1}),w(u_{2})\right>$ which is linear in $w(u_{1})$ and $w(u_{2})$, this means that $X(G,W)$ corresponds to the smooth divisor on $X(G\setminus e, W)$ with bundle $\mathcal{O}_{X(G\setminus e,W),u_{1}}(1)\otimes \mathcal{O}_{X(G\setminus e,W),u_{2}}(1)$. Now, using the adjunction formula, we have: $$K_{X(G,W)}=i^{*}(K_{X(G\setminus e,W)}\otimes \mathcal{O}_{X(G\setminus e,W),u_{1}}(1)\otimes \mathcal{O}_{X(G\setminus e,W),u_{2}}(1))$$ 
	By induction we know: $$K_{X(G\setminus e,W)}=\bigotimes_{v\in V} \mathcal{O}_{X(G\setminus e,W),v}(-n+\left| N(v)\right|)\otimes \mathcal{O}_{X(G\setminus e,W),u_{1}}(-1)\otimes \mathcal{O}_{X(G\setminus e,W),u_{2}}(-1)$$
	So we get $$K_{X(G,W)}=i^{*}(\bigotimes_{v\in V} \mathcal{O}_{X(G\setminus e,W),v}(-n+\left| N(v)\right|))=\bigotimes_{v\in V} \mathcal{O}_{X(G,W),v}(-n+\left| N(v)\right|)$$
\end{proof}

A particular interesting case is when $n>D$ where $D$ is the maximal degree of the tree, then we get that the anti-canonical bundle is ample, since all the twists appearing in the formula for the canonical bundle are negative.

\section{Singularities}
\label{sec:4}
\subsection{Rational Singularities of $X(G,W)$}
\label{subsec:4.1}
An application of the last part we prove:
\begin{theoremC}
\label{thmC}
    Let $G$ be a forest with maximal degree $D$ and $W$ a symplectic space with $\dim(W)\geq D+1$, then $X(G,W)$ has rational singularities.
\end{theoremC}
Since it is obviously sufficient to prove this for trees, assume for the rest of this section that $G$ is a tree with maximal degree $D$.

\begin{definition}
	$Y(G,W)$ is the subspace of $\PX{G}{W}\times W$ s.t. for all points $(l,w)\in Y$ we have $w \in l $.
\end{definition}

There are natural morphisms from $Y(G,W)$ to $X(G,W)$ and $\PX{G}{W}$ given by forgetting the lines or the vectors respectively, and denoted by $f$ and $p$ respectively. 

Notice that on $f^{-1}(O(G, W))$, $f$ is an isomorphism. Thus $f$ is a birational equivalence between $Y(G, W)$ and $X(G, W)$. In fact, we claim that $f$ is a resolution of singularities. We have to verify that $Y(G, W)$ is smooth and that $f$ is proper. 

The latter is clear since $f$ can be decomposed as a closed embedding to a product of a variety and a projective space, followed by a projection to the variety.

As for the former, one can easily describe $Y(G,W)$ as the total space of a bundle that is essentially the tautological bundle of $\PX{G}{W}$, more formally:
\begin{claim}
	If $p_{v}:\PX{G}{W}\rightarrow \mathbb{P}(W)$ is the projection on $w(v)$ and $T$ is the tautological bundle on $\mathbb{P}(W)$ then:
	$$Y(G,W)=\bigoplus_{v\in V} p^{*}_{v}(T)$$
\end{claim}  

It follows that all the fibers of $p$ are $\mathbb{A}^{V}$, so in particular $Y(G,W)$ is smooth if $\PX{G}{W}$ is (i.e. when $\dim(W) > D$).

Thus we see that indeed $f$ is a resolution of singularities for $X(G, W)$. To check if $X(G, W)$ has rational singularities, we need to see if $R^q f_* \mathcal{O}_{Y(G, W)} = 0$ for $q > 0$.

Now we look at the diagram:

$$\begin{matrix}
	Y(G,W) & \arrow{r}{f} & X(G,W) \cr
	\arrow{d}{p} & & 	\arrow{d}{C} \cr
	\PX{G}{W} & \arrow{r}{C} & pt \cr
	
\end{matrix}$$ 

Where the arrows to $pt$ are the obvious maps.

Now notice that the push-forward from $X(G,W)$ to $pt$ is the  global sections functor, and since $X(G,W)$ is affine, on quasi-coherent sheaves it is faithfully exact, so it's enough to show that $R^{q}(Cf)_{*}\mathcal{O}_{Y(G,W)}=0$ for $q>0$.

Further more since $p$ is a vector bundle projection and therefore affine, $p_{*}$ is also exact, together with the fact that $Cf=Cp$ its enough to show that $R^{q}C_{*}(p_{*}\mathcal{O}_{Y(G,W)})=H^{q}(p_{*}\mathcal{O}_{Y(G,W)})=0$ for $q>0$.

Now we use the fact that if $(E,p)$ is a vector bundle, then $p_{*}\mathcal{O}_{E}=\Sym(E^{*})$ where $\Sym(E^{*})$ is the sum of all symmetric products of $E^{*}$.

In particular in our case:
$$p_{*}\mathcal{O}_{Y(G,W)}=\Sym(\bigoplus_{v\in V} p^{*}_{v}(T)^{*})=\bigoplus_{i:V\rightarrow \mathbb{N}} \bigotimes_{v\in V} p_{v}^{*}(T)^{-i(v)} $$
So we only have to show that $H^{q}(\bigotimes_{v\in V} p_{v}^{*}(T)^{-i(v)} )=0$ for all $i:V\rightarrow \mathbb{N}$ and $q>0$. For this we only need the fact that this line bundle has no base points.
Indeed, we have the general claim:
\begin{claim}
	If $\mathcal{L}$ is a line bundle with no base points on a projective variety $X$, and $\omega_{X}$ is anti-ample, then $H^{q}(\mathcal{L})=0$ for $q>0$.
\end{claim}
In our case, it is obvious that our line bundles have no base points since the anti-canonical bundle of $\mathbb{P}(W)$ and $\mathcal{O}_{\mathbb{P}(W)}$ have no base points, so each $p_{v}^{*}(T)^{*}$ has no base points.

\begin{proof}
	Notice that $\mathcal{L}=\omega_{X}\otimes \omega_{X}^{*}\otimes \mathcal{L}$. Since $\omega_{X}^{*}$ is ample, and $\mathcal{L}$ has no base points, then $\omega_{X}^{*}\otimes \mathcal{L}$ is ample as well. Now Kodaira's vanishing theorem says that $H^{q}(\omega_{X}\otimes (\omega_{X}^{*}\otimes \mathcal{L}))=0$ for $q>0$ and we are done.
\end{proof}

We saw that if $\dim(W)>D$ then $\omega_{X}$ is anti-ample, and so using the claim we are done.

\subsection{Degeneration}
\label{subsec:4.2}
In \cite[Section 2]{AA1}, in order to prove $\mathrm{FRS}$ (Flat Rational singularity) for the $Def_{G,n}$, a degeneration method was used, which essentially comes down to proving that our graph varieties have rational singularities. These methods also work for graphs varieties and give the following claim:
\begin{definition} 
	Let $G=(V,E)$ be a graph, $M$ a finite set, and $w:V\rightarrow \mathbb{Z}^{M}$ a weighting of the vertices. We define $w:E\rightarrow \mathbb{Z}^{M}$ by $w(\{v_{1},v_{2}\})=w(v_{1})+w(v_{2})$. Suppose that for each edge, the sequence $w(e)$ has a maximum. We then have a splitting of the graph $G_{i}=(V,E_{i})$ for $i\in M$ by $E_{i}=\{e\in E \mid \max(w(e))=w(e)_{i}\}$ 
\end{definition} 

\begin{claim}\cite[Corollary 2.4.9]{AA1}
	Let $G=(V,E)$ be a graph, $M$ a finite set, $w:V\rightarrow \mathbb{Z}^{M}$ a weighting of the vertices, and $W_{i}$ for $i\in M$ vector spaces. If $X(G_{i},W_{i})$ has rational singularities for any $i\in M$ then $X(G,\oplus W_{i})$ has rational singularities.
\end{claim}

This allows us to split the graph at the price of lowering the dimension of the vector space, so we would like to split the graph using as little colors as possible. We will prove the following theorem:

\begin{theoremD}
\label{thmD}
	Let $G$ be a graph with maximal degree $D$. There is a splitting with $p_{4}(D)=\frac{D^{2}(D+1)^{2}}{2}-1$ colors, s.t. each $G_{i}$ is the disjoint union of edges and vertices.  
\end{theoremD}
Since we already showed that for an edge the variety has rational singularities, and for a disjoint union the variety is the product, this shows that for $\dim(W)\geq2p_{4}(D)$, $X(G,W)$ has rational singularities.   

The proof will be split into five stages:

1) Split the vertices of the graph to different heights, s.t. in each height we have a graph with lower maximal degree.

2) Weight each level using induction on the maximal degree.

3) Add new colors, and choose a color for each edge between different levels.

4) Add weights for the new colors, realizing the chosen colors for edges between different levels without disturbing the coloring of edges inside levels.

5) Show this works.    

\subsubsection*{Stage 1.}
Assume w.l.o.g. that the graph is connected. We choose some vertex $v_{0}$ and define $V_{i}=\{v\in V \mid \text{the distance between } v_{0} \text{ and } v \text{ is exactly } i\}$.

Notice the following things:

1) The vertices in $V_{i}$ are connected only to $V_{i-1},V_{i},V_{i+1}$.

2) The induced graph on $V_{i}$ has maximal degree $\leq D-1$.

\subsubsection*{Stage 2.}
By induction, each of the graphs $V_{i}$ can be weighted in $p_{4}(D-1)$ colors s.t. the graph for each color is a disjoint union of edges and vertices.

Denote the weighting on $V_{i}$ by $w_{i}$.

Add a constant if necessary so that for each $i$, $w_{i}(x_{i})\geq w_{i-1}(x_{i-1})+w_{i-2}(x_{i-2})+10$ for each triplet $(x_{i},x_{i-1},x_{i-2})\in V_{i}\times V_{i-1}\times V_{i-2}$.

\subsubsection*{Stage 3.}
Now add $2D$ new sets of colors $A_{1,1},A_{1,2},...,A_{1,D},A_{2,1},A_{2,2},...,A_{2,D}$ each of size $D^2$.

Let $A_i = \bigcup_k A_{i, k}$. From now on, when we write $A_{i,k}$ or $A_i$ we mean that $i$ is taken mod $2$. 

We now choose the colors for each edge between different levels in the following way:

First, we choose that between $V_{i}$ and $V_{i+1}$ all edges have colors in $A_i$.
Now for each level $V_{i}$, choose for each vertex a number $1\leq k \leq D$ s.t. neighbors in $V_{i}$ have different numbers, this is possible since the maximal degree in $V_{i}$ is $<D$.

Finally for each edge $\{u, v\}$ between $u\in V_{i}$ and $v\in V_{i+1}$ we choose a color in $A_{i,k}$ where $k$ is the number chosen for $v$, such that every other edge between $V_i$ and $V_{i+1}$ whose vertex in $V_i$ is a neighbor of $v$ has a different color. This is possible since we only have $D^2 - 1$ requirements and $D^2$ color to choose from.

\subsubsection*{Stage 4.}

Now we have to choose the weighting of the new colors.

For each vertex $v\in V_{i+1}$, and for every edge between him and $V_{i}$, we give the color of that edge on $v$ the weight $\max(w_{i+1}(v))+\max(w_{i}(w))+5$ where the first $\max$ is taken over all the colors, and the second is taken over all vertices $w\in V_{i}$ and the colors.
Further more for an edge between $v$ and $V_{i+2}$ we give the color of that edge on $v$ the weight $5$.

\subsubsection*{Stage 5.}

We claim that this works, first we show that the coloring inside $V_{i}$ remains the same.

Since we did not change the weighting of the colors we got by induction, it's enough to show that each edge in $V_{i}$ is colored in one of the original colors.

let $e=\{u,v\}$ be an edge, and look at a color we added. If it is in some $A_{i,k}$ then on $v,u$ it can have weight at most $5$, and obviously $10$ is less than what all the original colors have on $v,u$.

If the color is in $A_{i-1,k}$ then it can only have non-zero weight on one of these vertices by the way we chose colors for edges. The weight is $\max(w_{i}(v))+\max(w_{i-1}(w))+5$, but any color of the original colors has weight at least $\max(w_{i-1}(w))+10$, and so the maximal color of $v$ has weight on $e$ bigger then the new color has on $e$.

Now all that is left is to show that edges between levels are colored in the way we wanted.

Let $e=\{u,v\}$ be an edge between different levels. On the color we wanted for it, the weight is $\max(w_{i}(v))+\max(w_{i-1}(w))+5+5$. Any other color from the new colors has weight equal to either $5$ or $\max(w_{i}(v))+\max(w_{i-1}(w))+5$ but not both, by the way we choose the colors, and any color of the original colors has weight at most $\max(w_{i}(v))+\max(w_{i-1}(w))$, so we are done.

\subsubsection*{Trees}
We also claim that for trees we can have an even better result, if the maximal degree is $D$ we can do this with $D$ colors.

\begin{proof}
	Pick a leaf of the tree, by induction we have a weighting satisfying the condition on the rest of the tree. Look at the edges of the one neighbor of the leaf. Except for the leaf there are at most $D-1$ of those, so pick a color that none of them have, and give the leaf a sufficiently large weight on that color. This clearly satisfies the conditions.  
\end{proof}
\medskip
\printbibliography


	





\end{document}